\documentclass{article}

\usepackage[affil-it]{authblk}
\usepackage{url}

\usepackage{graphicx}
\usepackage[dvipsnames]{xcolor}

\usepackage{amssymb}
\usepackage{amsmath}
\usepackage{amsthm}

\usepackage{algorithm}
\usepackage{algpseudocode}

\newtheorem{theorem}{Theorem}[section]
\newtheorem{lemma}[theorem]{Lemma}
\newtheorem{corollary}{Corollary}[theorem]

\renewcommand{\to}[1][]{ \xrightarrow{#1} }

\title{The Forestry of Adversarial Totient Iterations}
\author{LUIS PALACIOS VELA, CHRISTIAN WOLIRD}
\affil{University of California, Riverside}
\date{}

\begin{document}

\maketitle

\begin{abstract}
   We give a closed-form expression for $\varphi(1+\varphi(2+\varphi(3+...+\varphi(n)$, where $\varphi$ is Euler's totient function. More generally, for an integer sequence $A=\{a_j\}$ we study the value of $A^\varphi(n)=\varphi(a_1+\varphi(a_2+\varphi(a_3+...+\varphi(a_n)$ when $A$ is the perfect squares or the perfect cubes. We show $A^\varphi(n)$ is bounded for all sequences considered.
   We also present the \textit{Arboreal Algorithm} which can sometimes determine a closed form of $A^\varphi(n)$ using tree-like structures.
\end{abstract}

\section{Adversarial Totient Iterations}

The expression
\begin{equation*}
  \mathbb{N}^\varphi(n)=\varphi(1+\varphi(2+\varphi(3+...+\varphi(n)  
\end{equation*}
serves as a guiding example throughout this paper since rigorously determining its exact value was the initial goal of the authors\footnote{Regarding notation, we will omit closing parenthesis in these sorts of expressions when possible for readability. 
We are also using a slight-of-hand here; ``$\mathbb{N}$" refers to the positive integers as a sequence, rather than as a set. }.
As a reminder, Euler's totient function $\varphi(x)$ counts the coprime positive integers $y\le x$. For instance,
\begin{equation*}
    \varphi(12)=|\{1,5,7,11\}|=4\quad\text{and}\quad \varphi(1)=|\{1\}|=1.
\end{equation*}
So we can evaluate, say, $\mathbb{N}^\varphi(3)$ as
\begin{equation*}
    \varphi(1+\varphi(2+\varphi(3)))=\varphi(1+\varphi(4))=\varphi(3)=2.
\end{equation*}
More intuitively, we think of this evaluation as a sequence in which the terms are alternately increased by a predetermined increment and then reduced by $\varphi$.
\begin{equation*}
    0
    \color{OliveGreen}{ \to[+3] }
    3
    \color{red}{ \to[\ \varphi \ ] }
    2
    \color{OliveGreen}{ \to[+2] }
    4
    \color{red}{ \to[\ \varphi \ ] }
    2    
    \color{OliveGreen}{ \to[+1] }
    3
    \color{red}{ \to[\ \varphi \ ] }
    2.
\end{equation*}

Accordingly, we call these \textit{adversarial totient iterations} since they are a sort of competition between $\varphi$ pulling the sequence down and some arbitrary increments (the positive integers here) pushing the sequence up.

\newpage

\begin{figure}[h]
    \centering
    \includegraphics[scale=0.3]{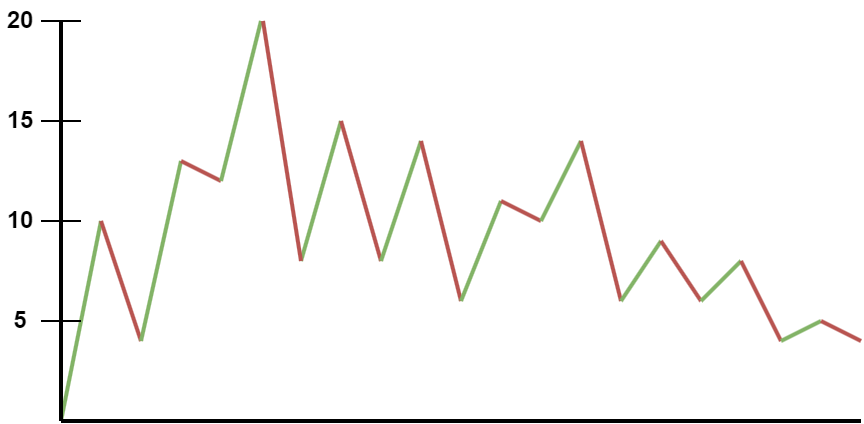}
    \caption{The adversarial iteration for $\mathbb{N}^\varphi(10)$.}
\end{figure}

While on terminology, we call $A^\varphi(n)$ in general the \textit{scoreboard function} since it tells us which mathematical entity ``won" the aforementioned competition. Roughly speaking, if $A^\varphi(n)$ is large, then the increment sequence $A$ successfully pushed the iteration upwards against $\varphi$. But if instead $A^\varphi(n)$ is small, then $\varphi$ pulled the sequence down against $A$. 

In fact, one can think of this ``competition" as a strange sort of infinite baseball game where the away team $A$ only bats and the home team $\varphi$ only fields. We declare $A$ the winner if they can increase their score arbitrarily and instead declare $\varphi$ the winner if they can keep $A$'s points below some finite bound. 

Accordingly, given a sequence $A$, we can define its \textit{scoreboard sequence}
\begin{equation*}
    A^\varphi 
    \quad=\quad 
    (A^\varphi(1),\quad A^\varphi(2),\quad A^\varphi(3),\quad ...\quad ).
\end{equation*}
where each term is like the score for one inning of our infinite baseball game. One of the authors' long-term aims is to determine for which sequences $A$ is the scoreboard sequence $A^\varphi$ bounded.

Returning to our guiding example, one can try a handful of examples and correctly guess that
\begin{equation*}
    \varphi(1+\varphi(2+\varphi(3+...+\varphi(n)
    =\begin{cases}
    1 & n=1,2 \\
    2 & n=3,4 \\
    4 & n\ge 5
\end{cases}
\end{equation*}

This formula is effectively a ``victory" for $\varphi$ against the positive integers since the scoreboard sequence is bounded:
\begin{equation*}
    \mathbb{N}^\varphi=(1,1,2,2,4,4,4,...).
\end{equation*}
This fact was first noticed by \emph{Manuel Alejandro Leal Camacho}, an author's classmate, who verified it computationally up to $n<10^5$ or so. 
Proving this rigorously however is a bit more work and takes up the next section.

\newpage

For now, we give an outline of our method. Starting with elementary observations of Euler's totient function $\varphi$, we obtain an inductive upper bound on the partial evaluations of the scoreboard function $\mathbb{N}^\varphi(n)$. In particular, it will follow that
\begin{equation*}
    \mathbb{N}^\varphi(n)\le 6.
\end{equation*}

Secondly, we create an algorithm to take repeated pre-images of $\varphi$ in a recursive manner. 
We call this the \textit{Arboreal Algorithm} (since its output is a collection of trees). 
The results, which we call \textit{totient trees}, can sometimes give us case equations for scoreboard functions.

\begin{figure}[h]
    \centering
    \includegraphics[scale=0.24]{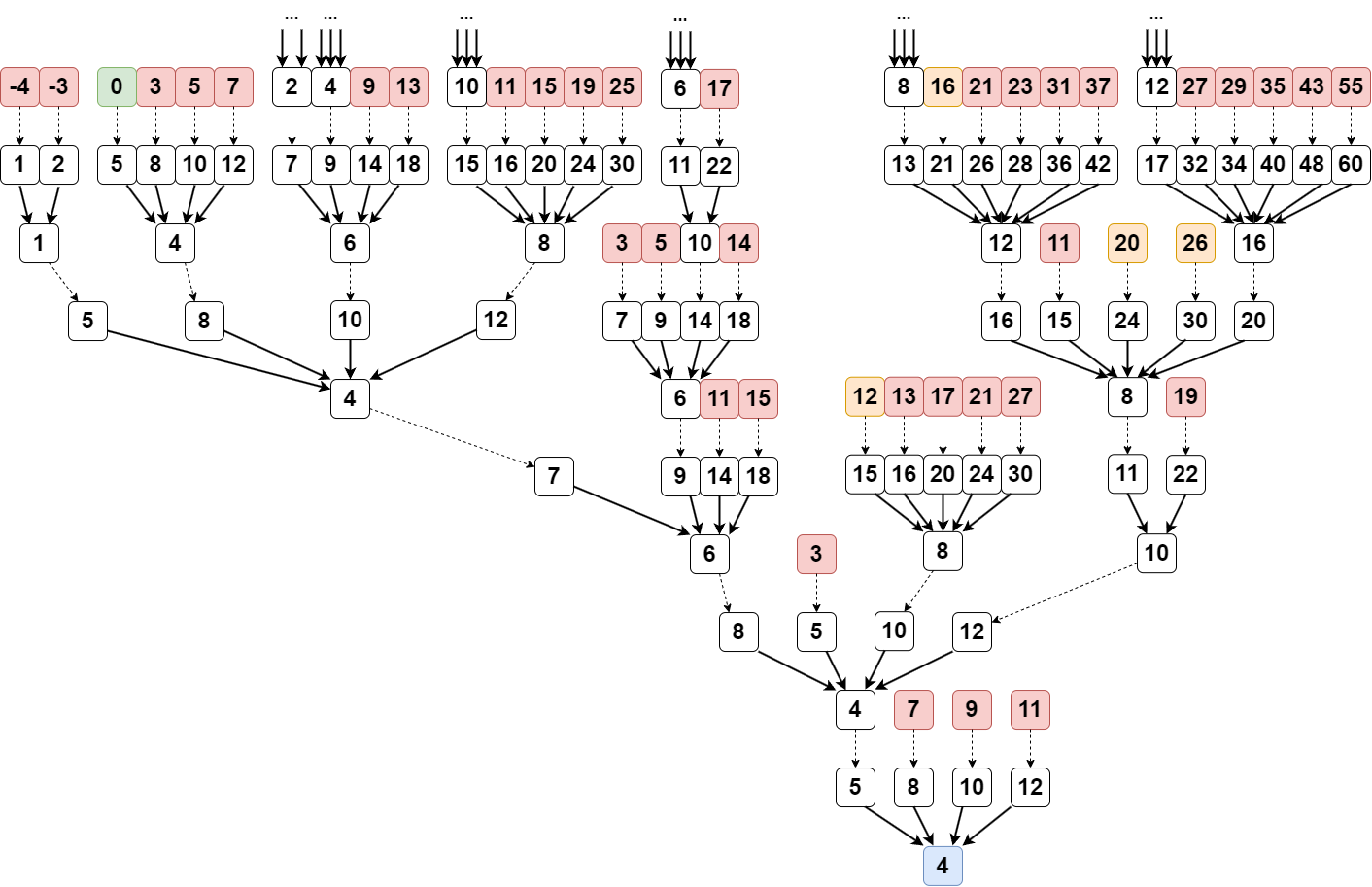}
    \caption{The totient tree produced by the Arboreal Algorithm for $\mathbb{N}^\varphi(n)=4$.}
\end{figure}

Once our opening example has been treated, we will go on to treat two more exotic variations. For instance, these same methods also yield the expression
\begin{equation*}
    \varphi(1+\varphi(4+\varphi(9+...+\varphi(n^2)
    =\begin{cases}
        1 & n=1 \\
        2 & n=2 \\
        4 & n=3 \\
        6 & n=4,5,6 \\
        22 & n=9,14,15,17,24,26,30,31,32,\\
         & 33,34,35,38,40,47,53,59,69 \\
        16 & \text{otherwise}
    \end{cases}
\end{equation*}
which can be seen as a ``victory" for $\varphi$ against the perfect squares. In fact, it would not be hard to generalize our methods to any sequence having polynomial growth and a positive lower density of even integers.

\newpage

\section{Euler's $\varphi$ vs. The Positive Integers}

Given a sequence $A=\{a_j\}$ we have to evaluate its scoreboard function
\begin{equation*}
    A^\varphi(n)=\varphi(a_1+\varphi(a_2+\varphi(a_3+...+\varphi(a_n).
\end{equation*}
from right to left as in the previous section's example. So, to continue our analysis, we need a better notation for the \textit{partial evaluations} obtained as we move from right to left. Let's define
\begin{equation*}
    A^\varphi(n, n):=0,
    \quad\text{and}\quad 
    A^\varphi(n, k-1):=\varphi\left(a_k + A^\varphi(n, k)\right).
\end{equation*}
Intuitively, $A^\varphi(n,k)$ is what you get by ignoring the outermost $k$ parenthesis in the former expression. Hence $A^\varphi(n,0)=A^\varphi(n)$.

In addition to this notation, we also need some facts about Euler's totient function. For all integers $n\ge 1$, we have:
\begin{itemize}
    \item $\varphi(n) \le n$ with equality only if $n=1$.
    \item $\varphi(2n)\le n$ with equality only if $n$ is a power of two.
    \item $\varphi(n)$ is even, when $n\geq 3$.
\end{itemize}
The first fact follows almost immediately from the definition of $\varphi$. 
The second fact follows from the observation that an odd number is coprime to $2n$ if it was also coprime of $n$.
% easy since $\varphi(2n)=n$ implies that \textit{every} odd number less than $n$ is coprime to $n$. 
Finally, the third fact follows quickly after noticing that if $n$ and $k$ are coprime, then so are $n$ and $n-k$. 

With all this background, we are ready for the first tool in our bag.

\begin{lemma}
    For all $n\ge 1$ and odd $k\ge 1$
    \begin{equation*}
        \mathbb{N}^\varphi(n,k)\le 2k+4.
    \end{equation*}
\end{lemma}
\begin{proof}
    We induct decreasingly on $k$. 
    If $n$ itself is odd then our base case,
    \begin{equation*}
        \mathbb{N}^\varphi(n,n)=0
    \end{equation*}
    clearly satisfies the bound $0<2n+4$. 
    On the other hand, if $n$ is instead even then our base case is
    \begin{equation*}
        \mathbb{N}^\varphi(n,n-1)=\varphi(n) 
        \le n \le 2(n-1)+4.
    \end{equation*}
    % which still satisfies the bound since
    % \begin{equation*}
    %     \varphi(n)\le n<2(n-1)+4.
    % \end{equation*}

    For our inductive case, we assume that $\mathbb{N}^\varphi(n,k)$ satisfies the bound for some $k\ge 3$ and must show from this that holds true for $\mathbb{N}^\varphi(n,k-2)$, i.e.,
    \begin{equation*}
        \mathbb{N}^\varphi(n,k-2)\le 2(k-2)+4=2k.
    \end{equation*}

    Firstly, we notice that $\mathbb{N}^\varphi(n,k)$ is, by definition, either $0$ or in the image of Euler's totient function. By the third fact before the lemma, this implies that $\mathbb{N}^\varphi(n,k)$ is either even or $1$. 
    In the case that $\mathbb{N}^\varphi(n,k)=1$, the induction holds since
    \begin{align*}
        \mathbb{N}^\varphi(n,k-2)
        &=\varphi(k-1+\mathbb{N}^\varphi(n,k-1))\\
        &\le k-1+\mathbb{N}^\varphi(n,k-1)\\
        &= k-1 +\varphi(k+\mathbb{N}^\varphi(n,k))\\
        &\le k-1 + k + \mathbb{N}^\varphi(n,k)\\
        &=2k.
    \end{align*}
    % \begin{equation*}
    %     \mathbb{N}^\varphi(n,k-1)
    %     \quad =\quad
    %     \varphi(k+1)\quad\le\quad k+1,
    % \end{equation*}
    % which implies, in turn, that
    % \begin{equation*}
    %     \mathbb{N}^\varphi(n,k-2)
    %     \quad=\quad
    %     \varphi(k-1+\mathbb{N}^\varphi(n,k-1))
    %     \quad\le\quad 
    %     k-1+\mathbb{N}^\varphi(n,k-1)\quad\le\quad 2k.
    % \end{equation*}
    If instead $\mathbb{N}^\varphi(n,k)$ is even, we must be a touch trickier. In this case, our inductive hypothesis tells us that
    \begin{equation*}
        \mathbb{N}^\varphi(n,k-1)
        \quad=\quad
        \varphi(k+\mathbb{N}^\varphi(n,k))
        \quad\le\quad
        k+\mathbb{N}^\varphi(n,k-1)
        \quad\le\quad 
        3k+4
    \end{equation*}
    
    Now, before we move on to $\mathbb{N}^\varphi(n,k-2)$, we can actually tighten this bound on $\mathbb{N}^\varphi(n,k-1)$ slightly. 
    We do this by recalling the ``only if" part of the first $\varphi$ fact listed before the lemma. 
    In our case $k\ge 3$, so we have rather obviously that $k+\mathbb{N}^\varphi(n,k)>1$. 
    It follows that the middle inequality of the former expression is \emph{strict}.
    Thus, we actually have
    $\mathbb{N}^\varphi(n,k-1)\le 3k+3$.

\begin{figure}[h]
    \centering
    \includegraphics[scale=0.24]{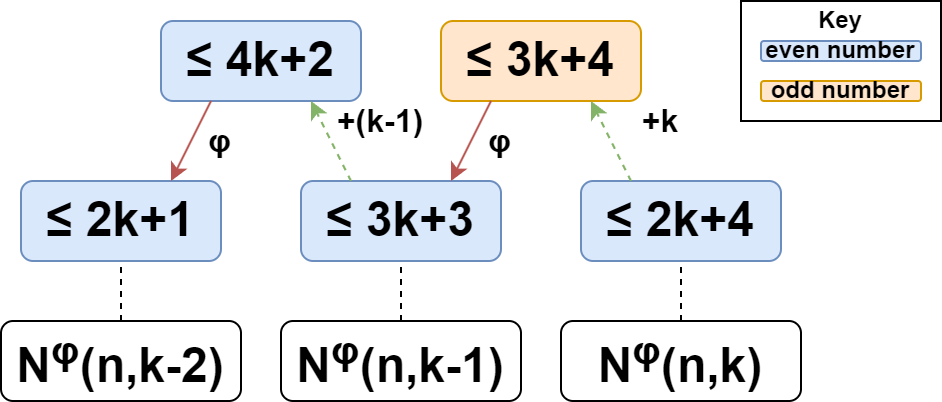}
    \caption{Progression of the upper bound on partial evaluations.}
\end{figure}

    From this, we can conclude
    \begin{gather*}
        \mathbb{N}^\varphi(n,k-2)
        \quad\le\quad \varphi(k-1+\mathbb{N}^\varphi(n,k-1))
        \quad\le\quad 
        \frac{1}{2}\left(k-1+\mathbb{N}^\varphi(n,k-1)\right)\\
        \le \frac{1}{2}(4k+2)=2k+1.
    \end{gather*}
    The factor of $\frac{1}{2}$ follows from the 2nd $\varphi$ fact listed before the lemma. 
    Using that both $k-1$ and $\mathbb{N}^\varphi(n,k-1)$ are even, so it is their sum, which at the very least, $\varphi$ cuts in half.

    Lastly, we point out that $\mathbb{N}^\varphi(n,k-2)$ is a even number bounded by an odd number, $2k+1$. 
    So we can actually tighten the bound by $1$ obtaining
    \begin{equation*}
        \mathbb{N}^\varphi(n,k-1)\le 2k=2(k-2)+4,
    \end{equation*}
    completing the induction.
\end{proof}

\newpage

The former proof was much more tedious than it really had to be. 
We could have gotten away with a bound of, say, $2k+8$ more easily. 
The remaining analysis would go through without too much extra trouble. 
We wanted to tighten down this first lemma as much as possible, to get everything out of it we could, and to test our bounds. 
We will be less picky in later proofs.

Now, speaking of ``the remaining analysis", we start by pointing out one result that follows from the proof we just did.

\begin{corollary}
    For all $n\ge 1$ and even $\ell \ge 0$
    \begin{equation*}
        \mathbb{N}^\varphi(n,\ell)\le 3\ell+6
    \end{equation*}
\end{corollary}
\begin{proof}
    In the previous proof, we showed that for a $k\ge 3$,
    \begin{equation*}
        \mathbb{N}^\varphi(n,k-1)\le 3k+3 = 3(k-1)+6.
    \end{equation*}
    By setting, $\ell=k-1$, we get the bound we want.
    If $n=k$, then clearly $\mathbb{N}^\varphi(n,n)=0$ satisfies the bound. If instead $k<n$, then $\mathbb{N}^\varphi(n,k+1)$ is defined and we have
    \begin{equation*}
        \mathbb{N}^\varphi(n,k+1)\le 2(k+1)+4=2k+6
    \end{equation*}
    from which it follows that
    \begin{equation*}
        \mathbb{N}^\varphi(n,k)
        \quad=\quad
        \varphi(k+1+\mathbb{N}^\varphi(n,k+1))
        \quad\le\quad 3k+7
    \end{equation*}
    And if $k+1+\mathbb{N}^\varphi(n,k+1)\not=1$, then this inequality is actually strict and we can tighten the bound to $\mathbb{N}^\varphi(n,k)\le 3k+6$. If instead $k+1+\mathbb{N}^\varphi(n,k+1)=1$, then we must have $k=0$ and $n=1$ in which case $\mathbb{N}^\varphi(1,0)=1\le 7$
    clearly satisfies the bound.
\end{proof}

\begin{corollary}
    For all $n\ge 1$, we have $\mathbb{N}^\varphi(n)\in\{1,2,4,6\}$.
\end{corollary}
\begin{proof}
    The previous lemma tells us that $\mathbb{N}^\varphi(n,1)\le 6$. It follows that 
    \begin{equation*}
        \mathbb{N}^\varphi(n)
        \quad=\quad
        \varphi(1+\mathbb{N}^\varphi(n,1))
        \quad\le\quad 1+\mathbb{N}^\varphi(n,1)
        \quad \le \quad 7
    \end{equation*}
    And since $\mathbb{N}^\varphi(n)$ is in the image of Euler's totient function $1$, $2$, $4$, and $6$ are the only possibilities.
\end{proof}

\begin{theorem}
    The exact values of the scoreboard function $\mathbb{N}^\varphi(n)$ are
    \begin{equation*}
        \mathbb{N}^\varphi(n)
        =\varphi(1+\varphi(2+\varphi(3+...+\varphi(n)
        =\begin{cases}
            1 & n=1,2 \\
            2 & n=3,4 \\
            4 & n\ge 5
        \end{cases}
    \end{equation*}
\end{theorem}
\begin{proof}
    Given the preceding corollary, we only have to precisely determine when $\mathbb{N}^\varphi(n)$ takes the values $1$, $2$, $4$, and $6$. In the preceding lemma, we move through the partial evaluations $\mathbb{N}^\varphi(n, k)$ downward, taking bounds. Here, we do roughly the opposite, moving upward through the partial evaluations, taking complete pre-images.

    We work over the case $\mathbb{N}^\varphi(n)=1$ with our bare hands. But since the analogous treatments of $\mathbb{N}^\varphi(n)=2$ and $\mathbb{N}^\varphi(n)=6$ are much more tedious (though not much harder) we elect to present them graphically in the next figure. 

    \newpage

\begin{figure}[h]
    \centering
    \includegraphics[scale=0.3]{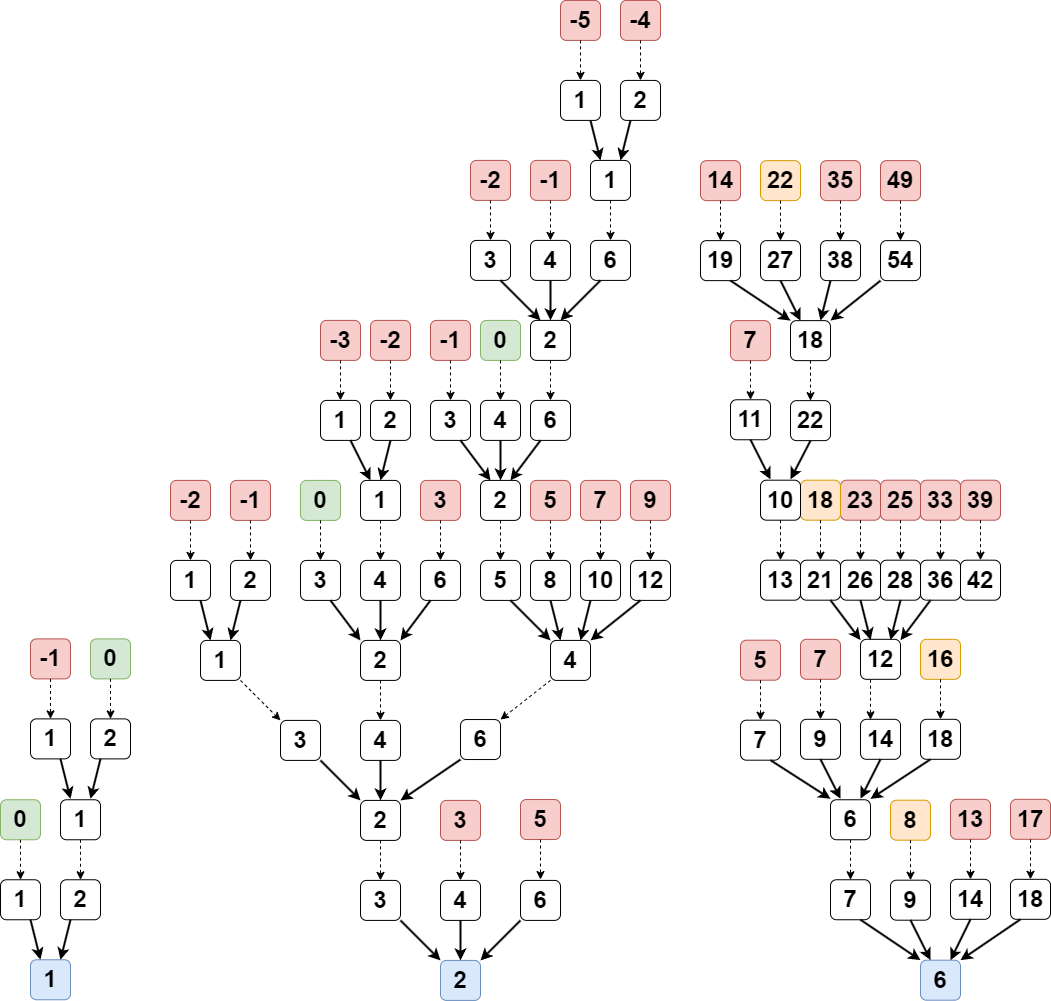}
    \caption{The totient trees corresponding to all possible values that the partial evaluations might take for $\mathbb{N}^\varphi(n)=1$, $\mathbb{N}^\varphi(n)=2$, and $\mathbb{N}^\varphi(n)=6$. Red nodes mark the numbers that are not in the image of $\varphi$. Orange nodes mark numbers that are ruled out by the previous upper bounds. And green nodes mark the legal values for $n$ we're looking for (which the authors like to imagine as fruit hanging from the trees).}
\end{figure}

    Thus for this first case, we suppose that we have 
    \begin{equation*}
        \mathbb{N}^\varphi(n)=\varphi(1+\mathbb{N}^\varphi(n,1))=1
    \end{equation*}
    for some $n$. We naturally ask: which are all the possible integers $x$ such that $\varphi(x)=1$? It's not hard to see that $x=1$ and $x=2$ are the only possibilities. So, plugging these into the former expression, we either have
    \begin{equation*}
        \mathbb{N}^\varphi(n,1)=0\quad\text{or}\quad \mathbb{N}^\varphi(n,1)=1
    \end{equation*}
    The first case would imply that $n=1$ since the partial evaluation $\mathbb{N}^\varphi(n,k)$ can only be zero if $n=k$. In the second case, we can rinse and repeat since we have
    \begin{equation*}
        \mathbb{N}^\varphi(n,1)=\varphi(2+\mathbb{N}^\varphi(n,2))=1
    \end{equation*}
    And it similarly follows that either 
    \begin{equation*}
        \mathbb{N}^\varphi(n,2)=-1\quad\text{or}\quad \mathbb{N}^\varphi(n,2)=0.
    \end{equation*}
    The first case is impossible. And the second case implies $n=2$. Thus, it follows that if $\mathbb{N}^\varphi(n)=1$, then either $n=1$ or $n=2$.

    This approach generalizes to a sort of recursive tree search as visualized in the preceding figure. The result is that $\mathbb{N}^\varphi(n)=2$ can occur only for $n=3,4$ and that $\mathbb{N}^\varphi(n)=6$ can never occur at all.

    It follows that $\mathbb{N}^\varphi(n)=4$ for all $n\ge 5$ and this gives us the case equation in the theorem statement.

\end{proof}

One detail we neglected in this previous proof is to verify that our pre-images of Euler's totient function are correct. That is, we skipped the following problem: given an integer $m$, find all integers $n$ such that $\varphi(n)=m$. Going forward, we call the set
\begin{equation*}
    \varphi^{-1}(n)=\{m\in\mathbb{N}:\varphi(m)=n\}
\end{equation*}
the \textit{totient fiber} of $n$. And we amend the lack of justification here by giving a totient fiber algorithm in the following section.

This turns out to be a delicate matter as well. We pointed out a few pages back that if $x$ is in the image of $\varphi$, then $x$ is either even or $1$. But the natural converse is not true. That is, there are even numbers which are not in the image of $\varphi$. The number $14$ for instance, has an empty totient fiber.

To close out this section, we take half a page to pause and ask...\\

\underline{\textbf{Where are we?}}: Our first main result here was to prove an upper bound on the scoreboard sequence, $\mathbb{N}^\varphi(n)\le 7$. The important parts of this proof were
\begin{enumerate}
    \item that $\varphi$ cuts even numbers in half,
    \item that the image of $\varphi$ is almost entirely even numbers,
    \item and that every other positive integer is even.
\end{enumerate}
In other words, we relied almost entirely on the density of even numbers within the positive integers. And these properties of even numbers all worked together to place a linear bound on the partial evaluations of our scoreboard sequence.

Then with the image of $\mathbb{N}^\varphi(n)$ restricted to finitely many values, we took totient fibers recursively to determine when exactly each value in the image could occur. We found that two such values, $1$ and $2$, occurred for finitely many $n$, one value $4$ occurred for infinitely many $n$, and another value $6$ didn't occur at all (its tree died without bearing fruit!). 

This last point has an interesting meaning for further exploration. It tells us that even if our upper bound on the image of a scoreboard function is optimal, that some possible values simply don't ``happen" to occur.

\newpage

\section{The Arboreal Algorithm}

 The authors were able to calculate the totient trees of $\mathbb{N}^\varphi(n)$ from the previous section with pen and paper. However, we'd also like to analyze the scoreboard functions $A^\varphi(n)$ for sequences other than the positive integers. And our approach becomes unmanageable by human methods in many cases (with some totient trees reaching hundreds fibers high before dying). So we turn to computers for what they do best.

 And if we want an algorithm for totient trees, then we'll first need an algorithm for totient fibers (i.e. pre-images of Euler's totient function $\varphi$). We use the method outlined by Dr. Richard Borcherds in his number theory video series [Bor22].

\begin{algorithm}
\caption{(Totient Fibers) Given an $m$, finds all $n$ such that $\varphi(n)=m$}
\begin{algorithmic}
\Procedure{TotientFiber}{$m$}
\If{$m=1$} \Comment{Special case}
    \State \textbf{yield} $1$
\EndIf
\State $C\gets \{1\}$ \Comment{Stores candidate values for $n$ (denoted $n'$)}
\For{$d\in $ divisors$(m)$}
    \If{is\_prime$(d+1)$}
        \State $p\gets d+1$ \Comment{Iterates over all possible prime divisors of $n$}
        \State $C_p\gets\{\}$ \Comment{New candidate values that are multiples of $p$}
        \For{$n'\in $ fiber}
            \State $k\gets 1$
            \While{$\varphi(n'p^k)\big|m$} \Comment{Checks if $k$ is too big}
                \If{$\varphi(n'p^k)=m$} \Comment{Do we have a value for $n$?}
                    \State \textbf{yield} $n'p^k$ \Comment{Yes: yield the value}
                    \If{$2\not|n'p^k$}
                        \State \textbf{yield} $2n'p^k$ \Comment{and its double if odd}
                    \EndIf
                \Else \Comment{No: add the value to our candidates}
                    \State $C_p\gets C_p\cup\ \{n'p^k\}$
                \EndIf
            \EndWhile
            \State $k\gets k+1$ \Comment{Increments the exponent of $p$}
        \EndFor
        \State  $C_p\gets C\cup C_p$ \Comment{Adds new candidates to the main list}
    \EndIf
\EndFor
\EndProcedure
\end{algorithmic}
\end{algorithm}

To explain this algorithm, we need two more facts about $\varphi$.
 \begin{itemize}
     \item For any $p$ prime and positive integer $k$ we have $\varphi(p^k)=(p-1)p^{k-1}$
     \item If $m$ and $n$ are coprime then $\varphi(mn)=\varphi(m)\varphi(n)$
 \end{itemize}
For the first fact, we notice that exactly $1$ in $p$ integers less than $p^k$ are multiples of $p$ so that we have $\varphi(p^k)=p^k-\frac{1}{p}p^k$. The second fact follows from the Chinese Remainder Theorem. Putting both facts together, we get an exact expression for $\varphi(n)$ modeled on the prime factorization of $n$.
\begin{equation*}
    n=\prod p_j^{r_j}\quad\Rightarrow\quad \varphi(n)=\prod (p_j-1)p_j^{r_j-1}
\end{equation*}

The important bit here is that if we are given $m=\varphi(n)$ and are asked to find all possible values of $n$, this formula tells us that every prime divisor $p_j$ of $n$ must be $1$ more than a divisor of $m=\varphi(n)$. 

For example, if we are told that $\varphi(n)=24$, then the prime factors of $n$ must be a subset of $\{2, 3, 5, 7, 13\}$ since these are the only primes which are $1$ more than a divisor of $24$. And from here one can check (rather tediously) that $n=35,39,45,52,56,70,72,78,84,$ or $90$ by trying all possible products of the former primes. This process is made explicit in the previous algorithm.

We should also note: this totient fiber algorithm was ran using SageMath [Ste24] and the authors found it to run roughly twice as fast when iterating over the divisors of $m$ in descending order rather than ascending (taking on average 4 seconds instead of 10 to compute the totient fibers of the first $10^4$ integers).

Now that we have an algorithm for totient fibers, we can create an algorithm for totient \textit{trees}. In practice, we generate multiple trees at a time and typically with extra code that gathers statistical data along the way. For simplicity here, we only give pseudocode for generating a single tree and finding the zeros (or fruit!) in its branches.

\begin{algorithm}
\caption{(Arboreal Algorithm) Given an integer $x$, a height $k$, a sequence of increments $A=(a_1, a_2, a_3, ...)$, and a sequence of upper bounds $B=(b_0, b_1, b_2, ...)$, finds all $n$ such that $A^\varphi(n)=x$ and $A^\varphi(n, k)\le b_k$.}
\begin{algorithmic}
\Procedure{Fruit}{$x, k, A, B$}
    \If{$k\le |A|$} \Comment{Checks for overgrowth}
        \If{$x=0$} \Comment{Have we found a fruit?}
            \State \textbf{yield} $k$ \Comment{Base case}
        \EndIf
        \If{$0<x\le b_k$} \Comment{Do we meet the upper bound?}
            \For{$y\in $ TotientFiber$(x)$}
                \For{$k'\in $ Fruit$(y-a_{k+1}, k+1, A, B)$}
                    \State \textbf{yield} $k'$ \Comment{Recursive case}
                \EndFor
            \EndFor
        \EndIf
    \Else
        \State \textbf{yield} ``Overgrown"
    \EndIf
\EndProcedure
\end{algorithmic}
\end{algorithm}

\newpage

To make sense of this code, we start with the reminder that each node in a totient tree corresponds to a possible value that some partial evaluation $A^\varphi(n, k)$ might take. And if one such value $y=A^\varphi(n, k)$ and another such value $x=A^\varphi(n, k-1)$ satisfy the recurrence we used to define these partial evaluations,
\begin{equation*}
    x=\varphi(a_k + y)
\end{equation*}
then we draw an edge downwards from $y$ to $x$. Although, in our totient tree diagrams, we actually drew in a bit more for completeness. We drew one node for $y$, another for $a_k+y$, and a third for $x=\varphi(a_k+y)$. Then we drew a dashed arrow down from $y$ to $a_k+y$ (representing the increment portion of our adversarial iteration). And we drew a solid arrow down from $a_k+y$ to $x$ (representing Euler's totient function).

To help guide intuition, we include here part of the totient tree for the scoreboard function obtained from the perfect squares $A=(1,4,9,...)$.

\begin{figure}[h]
    \centering
    \includegraphics[scale=0.35]{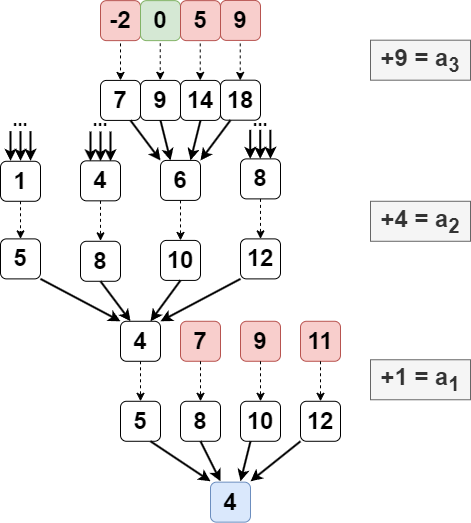}
    \caption{A portion of the totient tree corresponding to $A^\varphi(n)=4$ where $A=(1,4,9,...)$ is the perfect squares.}
\end{figure}

In Figure 5, the blue base node corresponds to $A^\varphi(n)=A^\varphi(n, 0)=4$. Similarly, the white ``$6$" node corresponds to $A^\varphi(n, 2)=6$. And the top green node corresponds to $A^\varphi(n, 3)=0$. We draw a dashed line from the green node down to $9$ since $A^\varphi(n, 3)+a_3=0+9=9$ And we draw a solid line from the $9$ to the $6$ since 
\begin{equation*}
    A^\varphi(n, 2)=\varphi\Big(a_3+A^\varphi(n, 3)\Big)=\varphi(9)=6
\end{equation*}
And all in all, the path from the green $A^\varphi(n, 3)=0$ to the blue $4$ represents the fact that 
\begin{equation*}
    A^\varphi(3)=\varphi(1+\varphi(4+\varphi(9)))=4
\end{equation*}

We think of any zeros as the ``fruit" of the totient tree. Hence, we have titled Algorithm 2 as the ``Fruit" procedure since it effectively yields the heights of the fruit in a given tree. This procedure is a recursive depth-first-search algorithm. When given the tuple $(x, k, A, B)$ we think of the procedure as being ``located" in the node representing $A^\varphi(n, k)=x$. The procedure then ``climbs" the tree by making recursive calls. These come from the recurrence relation
\begin{equation*}
    x=A^\varphi(n, k)=\varphi\Big(a_{k+1}+A^\varphi(n, k+1)\Big)
\end{equation*}
This is because for each element of the totient fiber $y\in \varphi^{-1}(x)$, we have a branch sprouting from the current node. And the new node resulting from such a branch represents the partial evaluation
\begin{equation*}
    A^\varphi(n, k+1)=y-a_{k+1}
\end{equation*}

Now, the red nodes in the former tree represent numbers with empty totient fibers. These empty fibers prune the tree back a bit. However, with this as our only pruning, nearly all totient trees would grow infinitely large. This is why we needed Lemma 2.1. It gives us an additional pruning method. Empty totient fibers prune back branches whose values grow too small. Our upper bound prunes back branches whose values grow too large. 

And so the type of pruning that a tree receives depends on how large, relatively speaking, its base is. For instance, $A^\varphi(n)=4$ is a relatively small value for the scoreboard function of the perfect squares. And accordingly, none of the nodes in the former diagram are killed by the upper bound (which we give in the next section). Oppositely, $\mathbb{N}^\varphi(n)=6$ is a relatively large value for the scoreboard function of the positive integers. And accordingly, the corresponding totient tree, given back in Figure 4, contained quite a few nodes pruned by our upper bound (marked there in orange).

Lastly, we should mention that in practice, we don't give the whole infinite sequence $A$ to the Arboreal Algorithm. Usually, the authors only give it the few hundred terms or so. If a tree grows so large that it exhausts these terms, we have the algorithm print some kind of ``overgrowth" warning. This lets us know that more terms are needed if we want any chance at knowing the fate of the tree in question. 

And that is the real difficulty of this sort of problem. For a general sequence $A$, we have no idea whether a given tree will grow forever like the tree for $\mathbb{N}^\varphi(n)=4$, or if it will die after five or six fibers like the tree for $\mathbb{N}^\varphi(n)=6$. In theory, there could be sequences for which, like $\mathbb{N}^\varphi(n)$, all but one tree die off. Except that even with optimal bounds of the sort given in Lemma 2.1, the other trees may not be totally pruned for millions or trillions of fibers, far beyond the reach of the authors' laptops. And our ability to give a closed form case expression, like that of Theorem 2.2, depends entirely on the dominance of a single tree in the totient forest ecosystem.

But so far, we have only explored the forest of the positive integers. The next sections will introduce new ecosystems with stranger behavior.

\newpage

\section{Euler's $\varphi$ vs. The Perfect Squares}

We built up a few tools in the previous sections for $\mathbb{N}^\varphi(n)$, the scoreboard function of the positive integers. Now it's time to see what else our tools can do. We start with the perfect squares $A=(1,4,9,...)$ giving us the scoreboard function
\begin{equation*}
    A^\varphi(n)=\varphi(1+\varphi(4+\varphi(9+...+\varphi(n^2)
\end{equation*}
as we have already hinted. 

Our methods here follow the same outline: placing an upper bound on the partial evaluations and using the Arboreal Algorithm to determine precise values. Here we omit repeated justification of the basic properties of Euler's totient function.

\begin{lemma}
    For all $n\ge 1$ and odd $k\ge 1$
    \begin{equation*}
        A^\varphi(n, k)\le 2k^2+14k+40
    \end{equation*}
    where $A=(1,4,9,...)$ is the perfect squares.
\end{lemma}
\begin{proof}
    We again induct downwards on $k$. The base cases are
    \begin{equation*}
        A^\varphi(n, n)=0
        \quad\text{and}\quad 
        A^\varphi(n, n-1)=\varphi(n^2)\le n^2\le 2(n-1)^2+14(n-1)+40
    \end{equation*}
    for $n$ odd and $n$ even respectively. Supposing for the inductive case that $A^\varphi(n,k)$ meets the bound, we have
    \begin{equation*}
        A^\varphi(n, k-1)=\varphi\Big(k^2+A^\varphi(n,k)\Big)\le 3k^2+14k+39
    \end{equation*}
    This in turn implies that
    \begin{gather*}
        A^\varphi(n, k-2)=\varphi\Big((k-1)^2+A^\varphi(n,k-1)\Big)\le \frac{1}{2}\Big((k-1)^2 + 3k^2+14k+39\Big)\\
        =2k^2+6k+20=2(k-2)^2+14(k-2)+40
    \end{gather*}
    also meets the bound. Once again the factor of $\frac{1}{2}$ appears because $(k-1)^2$ and $A^\varphi(n, k-1)$ are both even. The only exception to this is when $A^\varphi(n, k-1)=1$ in which case
    \begin{equation*}
        A^\varphi(n, k-2)=\varphi\Big((k-1)^2+1\Big)\le k^2-2k+2
    \end{equation*}
    still meets the bound.
\end{proof}

\begin{corollary}
    For all $n\ge 1$ and even $\ell\ge 0$
    \begin{equation*}
        A^\varphi(n, \ell)\le 3\ell^2+20\ell+57
    \end{equation*}
    where $A=(1,4,9,...)$ is the perfect squares.
\end{corollary}
\begin{proof}
    Setting $\ell=k-1$, the previous lemma tells us that
    \begin{equation*}
        A^\varphi(n, \ell)
        =\varphi\Big(k^2+A^\varphi(n, k)\Big)
        \le 3k^2+14k+40=3\ell^2+20\ell+57
    \end{equation*}
\end{proof}

\begin{corollary}
    For all $n\ge 1$, we have $A^\varphi(n)\le 57$ where $A=(1,4,9,...)$ is the perfect squares.
\end{corollary}
\begin{proof}
    Apply the previous corollary to $A^\varphi(n)=A^\varphi(n, 0)$.
\end{proof}

\begin{theorem}
    A complete description of $A^\varphi(n)$ is given by
    \begin{equation*}
        A^\varphi(n)=\varphi(1+\varphi(4+\varphi(9+...+\varphi(n^2)
        =\begin{cases}
            1 & n=1 \\
            2 & n=2 \\
            4 & n=3 \\
            6 & n=4,5,6 \\
            22  & n=9,14,15,17,24,26,30,31,32,\\
                & 33,34,35,38,40,47,53,59,69 \\
            16 & \text{otherwise}
        \end{cases}
    \end{equation*}
    where $A=(1,4,9,...)$ is the perfect squares.
\end{theorem}
\begin{proof}
    By the previous corollary, one only needs to run the Arboreal Algorithm on the positive integers less than $57$. In the resulting forest, every tree eventually died except the one corresponding to $A^\varphi(n)=16$. Five of the other trees bore fruit before dying. The former case equation effectively records the heights of these fruit and which trees they grew in.
\end{proof}

\begin{figure}[h]
    \centering
    \includegraphics[scale=0.35]{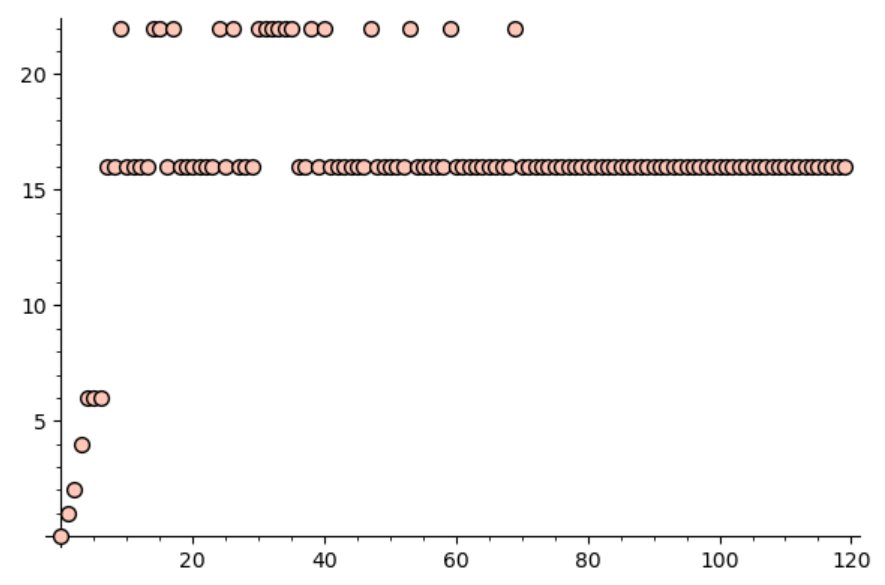}
    \caption{Graph of the scoreboard function $A^\varphi(n)$ where $A=(1,4,9,...)$ is the perfect squares and $1\le n< 120$}
\end{figure}

Interestingly, the tree corresponding to $A^\varphi(n)=22$ bore 18 fruit before finally dying over 100 fibers high. More generally, there is a sort of ``Goldilocks" behavior here. The values in the tree for $22$ were too large and the tree was pruned away by our upper bound. Oppositely, the values in the trees for $1$, $2$, $4$, and $6$ were too small and they were pruned away by empty totient fibers. But the values in the tree for $16$ were \textit{just right} and it grew upwards infinitely. 

\begin{figure}[h]
    \centering
    \includegraphics[scale=0.31]{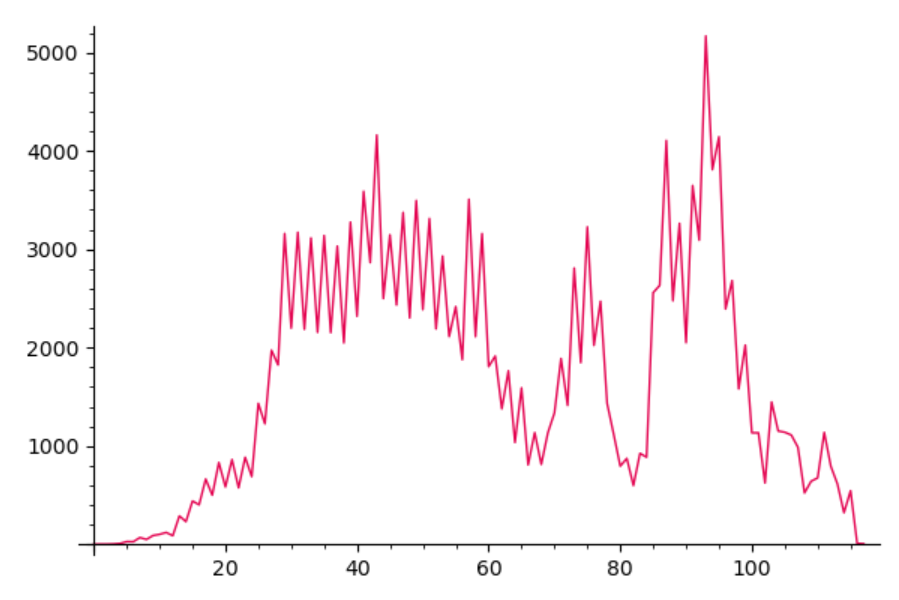}
    \caption{The size of the tree corresponding to $A^\varphi(n)=22$ as it grew. Horizontal axis is height. Vertical axis is nodes in the totient tree at that height.}
\end{figure}

To continue our arboreal metaphor, we can imagine all the values beneath our upper bound to be a sort of ``habitable zone" for the canopies of our trees. And for the totient forest of most scoreboard sequences, it appears that there is a particularly fertile ecological niche right in the center of the habitable zone that (so far) only one tree can occupy.

\begin{figure}[h]
    \centering
    \includegraphics[scale=0.31]{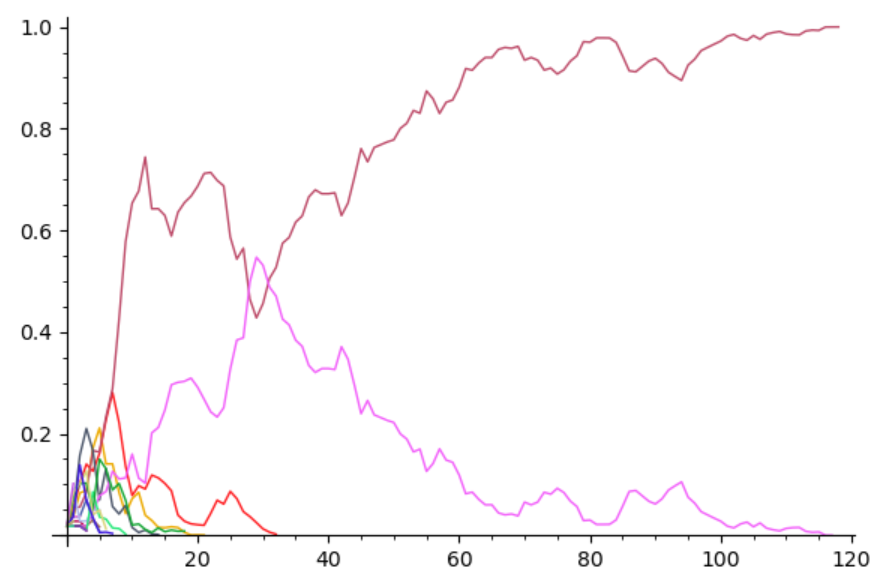}
    \caption{Canopy density by tree for the totient forest of the perfect squares. Horizontal axis is height. Vertical axis is percent of nodes at that height belonging to a given tree. The maroon line corresponds to $A^\varphi(n)=16$ and the pink line to $A^\varphi(n)=22$. The other lines are the smaller trees of the totient forest.}
\end{figure}

\section{Euler's $\varphi$ vs. The Perfect Cubes}

Now that we have solved the scoreboard sequence for the perfect squares, we'd like to go on to the perfect cubes. There are a few difficulties hitting us here.

Firstly, totient trees we are now dealing with are even bigger than those for the perfect squares. The ones we must inspect are hundreds or thousands of fibers high and contain thousands (if not millions) of nodes. Since picturing them as trees is difficult at this magnitude, we will instead plot statistical data from here on out. In the previous section, for instance, Figure 7 plotted the number of nodes which one particular totient tree had at each height. And Figure 8 plotted the percentage of nodes at a given height belonging to each tree in a forest.

The second difficulty is that, previously, every totient tree but one died. This does not seem to be the case with the perfect cubes. Instead, \textit{seven} totient trees manage to occupy the ``habitable zone" without suffocating each other\footnote{The authors verified that all seven of these trees survived up to $10000$ fibers high.}. This might mean that the number of trees which can survive is closely tied to the growth rate of the sequence in question. 

\begin{figure}[h]
    \centering
    \includegraphics[scale=0.4]{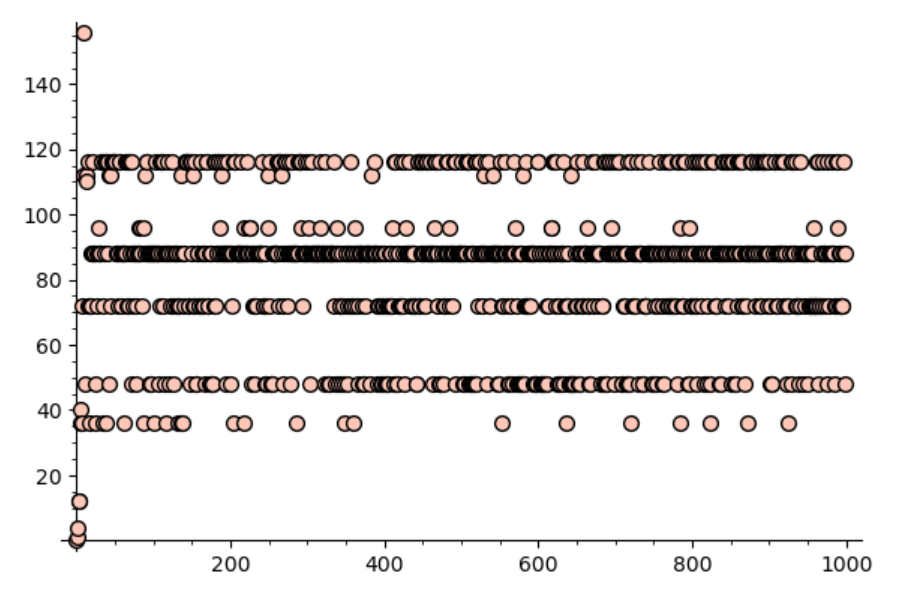}
    \caption{The scoreboard function $A^\varphi(n)$ where $A=(1,8,27,...)$ is the perfect cubes. The horizontal stretches of dots each distinguish a totient tree that survived to the height of $1000$ fibers.}
\end{figure}

Accordingly, we might conjecture that the scoreboard function of a linear or quadratic sequence eventually stabilizes to a single value. This would correspond to all but one tree surviving in the corresponding totient forest. On the other hand, the scoreboard function of a sequence with cubic growth rate might oscillate (seemingly randomly) between a handful of values. This would correspond to the larger habitable zone allowing the coexistence of multiple trees, each endlessly bearing their own fruit.

\newpage

And so, we omit the expected proofs for this section. This is partly 
% because some of the corresponding results for perfect powers and polynomial sequences in general are given in the following section. But this is also partly 
because the analogous theorem for perfect cubes is much weaker. Because multiple totient trees survived (as far as the authors were able to compute) we weren't able to produce a closed form case expression for the perfect cubes as we could for the positive integers and perfect squares. The closest we can get is
\begin{equation*}
    \varphi(1+\varphi(8+\varphi(27+...+\varphi(n^3)
    =\begin{cases}
        1 & n=1 \\
        4 & n=2 \\
        12 & n=3,4 \\
        36 &  <1\% \text{ chance}  \\
        40 & n=6 \\
        48 & \approx 14\% \text{ chance}  \\
        72 & \approx 23\% \text{ chance}  \\
        88 & \approx 35\% \text{ chance} \\
        96 & \approx 3\% \text{ chance}  \\
        110 & n=13 \\
        112 & \approx 2\% \text{ chance}  \\
        116 & \approx 23\% \text{ chance}  \\
        156 & n=9 \\
    \end{cases}
\end{equation*}

These percent chances are estimates based on the first $6000$ terms of the scoreboard sequence of the perfect cubes.

\begin{figure}[h] 
    \centering
    \includegraphics[scale=0.35]{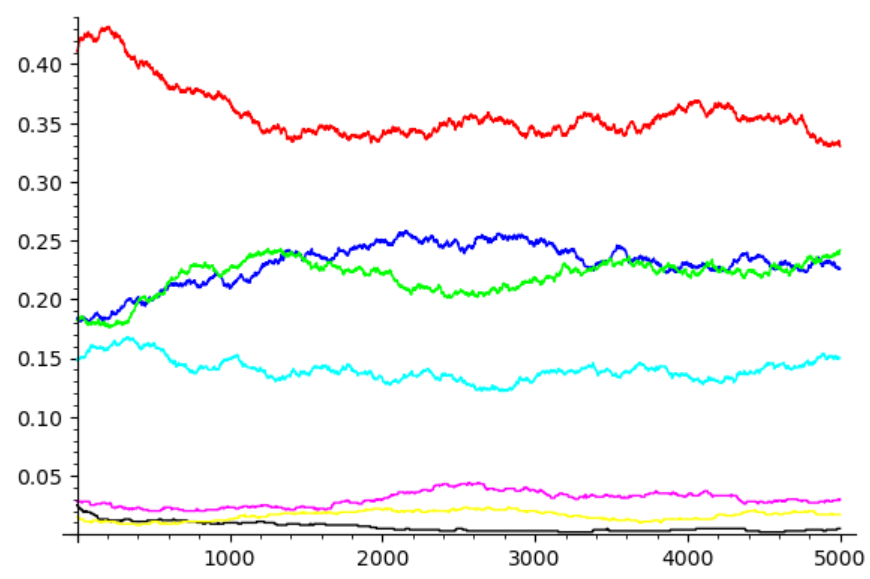}
    \caption{The percentage of fruit that each tree produces in the totient forest of the perfect cubes. Calculated over a rolling window of $1000$ totient fibers. The colors red, green, blue, cyan, magenta, yellow, and black correspond to the trees for $A^\varphi(n)=88$, $ 116$, $ 72$, $ 48$, $ 96$, $ 112,$ and $36$ respectively.}
\end{figure}

\newpage

\section{Where are we (going)?}

For an integer sequence $A=\{a_j\}$, we introduced the \textit{scoreboard function} of $A$,
$$A^\varphi(n)=\varphi(a_1+\varphi(a_2+\varphi(a_3+...+\varphi(a_n)$$
which can be seen as an adversarial iteration between Euler's $\varphi$ function and the sequence $A$. Long-term, the authors would like closed-form expressions and upper bounds on $A^\varphi(n)$ for various classes of sequences. 

We achieved a closed-form expression when $A$ is the positive integers or the perfect squares. This was in part done by placing inductive upper bounds on the adversarial iterations using the number theoretic properties of Euler's totient function. Additionally, we used the \textit{Arboreal Algorithm} to find closed form expressions for bounded scoreboard sequences by simulating the growth of tree-like structures.

Then, as a counter-example, we gave numeric evidence that for sequence of the perfect cubes $A=(1,8,27,...)$, the scoreboard function $A^\varphi(n)$ may not admit a closed-form expression. Essentially, we suggest that the only way to compute $A^\varphi(n)$ in the case of cubes, is by evaluating the recursive definition of the scoreboard function for individual values of $n$.

Now the methods used here for the positive integers and perfect squares will generalize to any sequence with 
\begin{itemize}
    \item Polynomial growth rate
    \item Positive lower density of even numbers
\end{itemize}
These seem to be the only essential properties used in our deductions. In particular, we expect that any \textit{integer polynomial sequence }
$$A=\Big(f(1),\ f(2),\ f(3),\ ...\Big)$$
given by an integer polynomial $f\in \mathbb {Z}[x]$ will admit an upper bound so long as it is not entirely valued on odd numbers (such odd-valued polynomials account for $25$\% of $\mathbb Z[x]$ in a reasonable sense). 

So the most natural sequence for which the authors have no current hope of generalizing the prior methods is the positive odd integers $A=(1,3,5,7,...)$. From the calculation of small values, the corresponding scoreboard sequence 
$$A^\varphi(n)=\varphi(1+\varphi(3+\varphi(5+...\varphi(2n-1)$$
appears to be bounded. But we have no way to prune back the totient trees and thus no hope of applying any variation of the Arboreal Algorithm.

One can at least show that the scoreboard sequences of a bounded sequence is itself bounded. If this were not so, one could use the partial evaluations to bound the prime gaps (which, of course, are unbounded). It may also be possible to show that certain slow growing sequences, like $A(n)=\lfloor\log\log(n)\rfloor$, have bounded scoreboard sequences regardless of how many even numbers they contain. The authors' hope is that if $A^\varphi(n)$ grows large for a very slow-growing $A$,  one could use the partial evaluations to place a lower bound on the asymptotic density of the primes (hopefully a bound large enough to contradict known upper bounds).

To close out, we leave the reader with an observation about exponential sequences. There's a sense, given in [Bor22], that the expected value of Euler's totient function is
$$\mathbb E[\varphi(n)]=\frac{6}{\pi^2}n$$
This means that iteratively applying Euler's totient function looks very much like exponential decay at a large scale. Thus, we expect roughly half of exponential sequences (those with growth rate $<\frac{\pi^2}{6}$) to have bounded scoreboard sequences and the other half (those with growth rate $\ge \frac{\pi^2}{6}$) to have unbounded scoreboard sequences. It would be interesting, in particular, to know whether the scoreboard sequence of the Fibonacci numbers is bounded.

\section*{References}

\noindent [Bor22] Richard Borcherds, \textit{Introduction to number theory lecture 14. Euler's totient function}, YouTube 2022\\

\noindent [Ste24] William A. Stein et al. Sage Mathematics Software (Version 10.4), The Sage Development Team, 2024, http://www.sagemath.org.

\end{document}